\documentclass[12pt,a4paper,reqno, oneside]{amsart}
\usepackage{lmodern}
\usepackage{a4wide}
\usepackage[english]{babel}
\usepackage{microtype}
\usepackage{amssymb}
\usepackage{bbm}

\usepackage{amsmath, amsthm}
\usepackage{mathtools}
\usepackage[dvipsnames]{xcolor}
\usepackage{tikz}
\usetikzlibrary{positioning}
\usetikzlibrary{intersections}
\usetikzlibrary{calc}
\usetikzlibrary{decorations.markings}
\usetikzlibrary {shapes.geometric} 

\usepackage[disable]{todonotes} 
\usepackage[pdfencoding=unicode]{hyperref}
\usepackage{bookmark} 
\usepackage[noabbrev]{cleveref}
\usepackage[textformat=period]{caption}
\usepackage{subcaption}
\usepackage{esdiff} 
\usepackage[draft]{listofsymbols}
\usepackage{chngcntr} 

\hypersetup{
	pdfauthor = {Judith Ortmann},
	colorlinks,
	linkcolor={red!80!black},
	citecolor={green!50!black},
	urlcolor={blue!80!black},
}

\usepackage{xurl}

\usepackage{graphicx}
\usepackage{csquotes}

\newtheorem{theorem}{Theorem}[section]
\newtheorem{prop}[theorem]{Proposition}
\newtheorem{lemma}[theorem]{Lemma}
\newtheorem{cor}[theorem]{Corollary}

\theoremstyle{definition}

\newtheorem{remark}[theorem]{Remark}

\numberwithin{equation}{section}

\newcommand{\PP}{\mathbb{P}}
\newcommand{\NN}{\mathbb{N}}
\newcommand{\ZZ}{\mathbb{Z}}
\newcommand{\QQ}{\mathbb{Q}}
\newcommand{\RR}{\mathbb{R}}
\newcommand{\CC}{\mathbb{C}}
\newcommand{\FF}{\mathbb{F}}
\newcommand{\GG}{\mathbb{G}}
\newcommand{\AAA}{\mathbb{A}}

\newcommand{\calO}{\mathcal{O}}
\newcommand{\Adele}{\mathbf{A}}

\newcommand{\dirichlet}{\varphi}
\newcommand{\ringofintegers}{\mathcal{O}_K}
\newcommand{\ringofintegersw}{\mathcal{O}_\omega}
\newcommand{\ringofintegerst}{\mathcal{O}_t}

\newcommand{\ringofintegerstinv}{\mathcal{O}_{t^{-1}}}

\newcommand{\absvalue}[1]{\left\lvert #1 \right\rvert}

\newcommand{\vol}{\mathrm{vol}}

\renewcommand{\O}{\mathcal{O}}
\newcommand{\one}{\mathbbm{1}}

\DeclareMathOperator{\Real}{Re}
\DeclareMathOperator{\Image}{Im}
\DeclareMathOperator{\Br}{Br}

\DeclareMathOperator{\inv}{inv}
\DeclareMathOperator{\Res}{Res}

\title{Rational points in a family of conics over $\FF_2(t)$}

\usepackage{color}

\newcommand{\dan}[1]{{\color{blue} \sf $\clubsuit\clubsuit\clubsuit$ Dan: [#1]}}

\author{Daniel Loughran}
\address{
Department of Mathematical Sciences \\
University of Bath \\
Claverton Down \
Bath\\ 
BA2 7AY\\
UK.}

\author{Judith Ortmann}
\address{Institut f\"ur Algebra, Zahlentheorie und Diskrete Mathematik, Leibniz Universit\"at Hannover, Welfengarten 1, 30167 Hannover, Germany}

\subjclass[2010]
{14G05 (primary), 
14F22 
(secondary).}

\begin{document}
	
\begin{abstract}
	Serre famously showed that almost all plane conics over $\QQ$ have no rational point.
	We investigate versions of this over global function fields,
	focusing on a specific family of conics over $\FF_2(t)$ which illustrates new behaviour.
	We obtain an asymptotic formula using harmonic analysis, which requires a Tauberian
	theorem over function fields for Dirichlet series with branch point singularities.
\end{abstract}

\setcounter{tocdepth}{1}
\maketitle
	
{\hypersetup{hidelinks}
	\tableofcontents}

	\thispagestyle{empty}
\section{Introduction}
Consider a morphism  $\pi:V \to \PP^n$ of smooth projective varieties over a  global field $K$ with geometrically integral generic fibre. We view this as a family of varieties parametrised by the fibres $V_x = \pi^{-1}(x)$.  A natural question is how many of the varieties in the family have a rational point. This is encapsulated via the counting function
	\[  \#\lbrace x\in \PP^n(K) : H(x) \leq B, ~ x\in\pi(V(K))\rbrace,\]
where $H$ is the usual naive height on projective space.
Serre \cite{serre_90} was amongst the first to consider such problems over $\QQ$. He phrased his results in terms of Brauer group elements, but was motivated by families of conics having a rational point. Serre proved  upper bounds in this case and asked whether his bounds were sharp. This question has attracted considerable interest and spawned many further research directions; see for instance \cite{bright_browning_loughran,browning_dietmann,guo_95,hooley_93,hooley_07,loughran_18,poonen_voloch}. 
In particular Loughran and Smeets \cite{loughran_smeets} have generalized the work of Serre from conics to other families of varieties and stated a conjecture on the asymptotic behaviour of the counting function.
	
In this article, we consider the same kind of problems of Serre and Loughran--Smeets but over global function fields $K$. We study whether the conjecture of Loughran and Smeets also holds in the global function field case. 
However, it turns out that there is new phenomena over global function fields, and we focus on an explicit example which demonstrates this.
	
\subsection{Main result}
Consider the global function field $K=\FF_2(t)$. 
Let $\Omega_K$ denote the set of places of $K$. 
Consider the following conic bundle surface
\[C\colon\,\, x_0^2 + x_0x_1 + yx_1^2 = tx_2^2\quad\subseteq \AAA^1_K\times\PP^2_{K}.\]
For fixed $y\in \AAA^1_K$, we denote by $C_y$ the conic given by the fibre of the morphism $\pi\colon C\to\AAA^1_K$ at the point $y$.
We are interested in the counting function
\begin{equation} \label{def:N(B)}
	N(\AAA^1_K,\pi,B) = \#\lbrace y\in K : H(y) = B, ~ C_y(K) \neq \emptyset\rbrace.
\end{equation} 
Here $H\colon \AAA^1_K\rightarrow\RR_{>0}$ is the usual naive height given by
\[ H(y) = \prod_{\omega\in\Omega_K} H_\omega(y) = \prod_{\omega\in\Omega_K} \max\lbrace 1,\absvalue{y}_\omega\rbrace.\]
Due to the height taking restricted values over function fields, we ask for equality of the height with $B$, rather than inequality. We consider this counting function as $B \to \infty$, where $B$ is restricted to lie in the set $\lbrace 2^M: M\in \NN_{\ge0}\rbrace$.  Our main result is as follows.
	
\begin{theorem}\label{main_theorem_rtl_points}
	\[N(\AAA^1_K,\pi,B) = \frac{cB^{2}}{(\log B)^{1/2}} + O\left(\frac{B^2}{(\log B)^{3/2}}\right),\]
	where
	$$c =4 \cdot \left(\frac{2(\log 2)}{\pi}\right)^{1/2}\prod_{\omega \in \Omega_K}\left(1-\frac{1}{2^{\deg \omega}}\right)^{1/2} c_\omega,$$
	with $c_\omega = \frac{1}{2}\left(\frac{1-2^{-2}}{1-2^{-1}}\right) = 3/4$ for $\omega \in \{t,t^{-1}\}$
and
$$c_\omega = 1 +  \frac{1}{2}
	\left(\frac{1}{2^{\deg\omega}}\left(1 + 
	\frac{1-2^{-\deg\omega}}{2^{\deg \omega}-1}\right) 
	- \frac{1}{2^{2\deg\omega -1}} \right),
$$
 for $\omega \notin \{t,t^{-1}\}$.
\end{theorem}

The Euler product in Theorem \ref{main_theorem_rtl_points} is easily seen to be absolutely convergent. 	
 As $\PP^1$ has approximately $B^2$ points of height equal to $B$ \cite[Thm.~3.11]{pey12}, \Cref{main_theorem_rtl_points} shows that $0\%$ of the varieties in the family $C$ contain a rational point.	We detect whether a conic in the family has a rational point using the \emph{Hasse principle}; this says that a conic has a $K$-point if and only if it has a $K_\omega$-point for all $\omega\in\Omega_K$.  To count we use harmonic analysis applied to the height zeta function
\begin{equation}\label{def_height_zeta_function}
	Z(s) = \sum_{\substack{y\in K\\ C_y(K)\neq \emptyset}}H(y)^{-s}, \quad \Real(s) > 2.
\end{equation}
We apply Poisson
summation to write this as a sum over additive adelic characters.
Usually, one aims for a meromorphic continuation on some half-plane with a unique pole of largest real value. However in our case there is a \textit{branch point} singularity, rather than a pole. As such to obtain the asymptotic formula we require a Tauberian theorem for branch point singularities. Over $\QQ$ such a Tauberian theorem has been obtained by Delange \cite{Del54}. We use a Tauberian theorem over function fields for this setting (Theorem \ref{Tauberian_Theorem}).
	
The harmonic analysis approach was one of the first successful attempts for proving Manin's conjecture \cite{fmt89}. Our method relies on harmonic analysis on the additive group $\GG_a$, which was first applied to Manin's conjecture over number fields in \cite{batyrev1998toric,chambert-loir_tschinkel_2002,chambert-loir_tschinkel_2000}, building on the case of toric varieties \cite{batyrev_tschinkel_95,batyrev_tschinkel_96,batyrev1998toric}

\subsection{Comparison with results over number fields}\label{section:comparison of results}
We compare Theorem \ref{main_theorem_rtl_points} with the known results in the literature by Serre \cite{serre_90}, and Loughran and Smeets \cite{loughran_smeets} over number fields, and point out new phenomena arising over global function fields.

\subsubsection{Specialisation of Brauer group elements}
We first compare our situation with that of Serre \cite{serre_90}. Recall that there is a correspondence between conics and quaternion algebras over a field. This correspondence even applies over fields of characteristic $2$, providing one interprets quaternion algebras correctly. Let $k$ be a field of characteristic $2$. Recall from Artin--Schreier theory that we have an isomorphism $H^1(k,\ZZ/2\ZZ) \cong k/\wp(k)$ where $\wp: a \mapsto a^2 - a$ is the Artin--Schreier map. Here $a \in k$ corresponds to the quadratic extension given by $x^2 - x - a$. Moreover by Kummer theory we have $H^1(k,\mu_2) = k^\times/k^{\times 2}$ for fppf cohomology. Therefore via the cup product we obtain a well-defined Brauer group element
\begin{equation} \label{eqn:alpha}
	\alpha = y \cup t \in \Br \AAA^1_K.
\end{equation}
This is the Brauer group element relevant to our situation, namely for $y \in K$ we have $\alpha(y) = 0 \in \Br K$ if and only if $C_y(K) \neq \emptyset$ (see Proposition \ref{equivalence_rtl_points_local_symbol}). Note that the existence of such a Brauer group element is special to positive characteristic, since $\mathbb{A}^1$ has constant Brauer group in characteristic $0$ \cite[Thm.~6.1.1]{Brauer}.

In \cite{serre_90} Serre obtains upper bounds over $\QQ$ for the number of rational points of bounded height in projective space for which a given Brauer group element specialises to $0$. His upper bound is phrased in terms of the \textit{residue} of the Brauer group element along codimension $1$ points.
However the residue is only defined \cite[Def~1.4.11]{Brauer} for so-called \emph{tame} Brauer group elements. Our element $\alpha$ is not tame, therefore the residue is not defined in this case, so the framework of Serre does not apply.

\subsubsection{Family of varieties}
We next consider the framework from \cite[Thm.~1.2]{loughran_smeets}. Let $V$ be a proper, smooth algebraic variety over a number field $K$ equipped with a dominant morphism $\pi: V\to\PP^n_K$ with geometrically integral fibre. Let  
	\[N_{\mathrm{loc}}(\PP^n_K,\pi,B) = \#\lbrace x\in\PP^n_K(K): x\in \pi(V(\Adele_K)),~ H(x)\le B\rbrace\]
be the number of varieties in the family that are everywhere locally solvable, where $H$ is the usual height on $\PP^n_K$, and $\Adele_K$ denotes the adeles of $K$. Then \cite[Thm.~1.2]{loughran_smeets}  says that
\begin{equation} \label{eqn:LS_upper_bound}
	N_{\mathrm{loc}}(\PP^n_K,\pi,B) \ll \frac{B^{n+1}}{(\log B)^{\Delta(\pi)}},
\end{equation}
where $\Delta(\pi) = \sum_{D\in(\PP^n_K)^{(1)}}(1-\delta_D(\pi))$ with
\[\delta_D(\pi) = \frac{\#\{\gamma\in\Gamma_D(\pi): \gamma~\text{acts with a fixed point on}~I_D(\pi)\}}{\#\Gamma_D(\pi)}.\]
Here $I_D(\pi)$ is the set of geometrically irreducible components of multiplicity one of the fibre of $\pi$ over $D$, and $\Gamma_D(\pi)$ is a finite group through which the action of $\mathrm{Gal}(\overline{\kappa(D)}/\kappa(D))$ on $I_D(\pi)$ factors with $\kappa(D)$  the residue field of the codimension one point $D$.
	
Moreover \cite[Conj.~1.6]{loughran_smeets} conjectures that \eqref{eqn:LS_upper_bound} is sharp when one assumes that at least one smooth fibre of $\pi$ is everywhere locally solvable and that the fibre of $\pi$ over every codimension one point of $\PP^n_K$ has an irreducible component of multiplicity one.
	
To compare these results with our situation, we first have to construct a proper model $\widetilde{C}$ of our surface $C$. Then, when we replace $H(x)\le B$ by $H(x) = B$, we can identify $N_{\mathrm{loc}}(\PP^1_K,\widetilde{\pi},B)$ with $N(\PP^1_K,\widetilde{\pi},B)$,	where we consider the map $\widetilde{\pi}: \widetilde{C}\to\PP^1_K$. Thus, by looking at our main result \Cref{main_theorem_rtl_points}, we would expect $\Delta(\widetilde{\pi}) = 1/2$.
To construct $\widetilde{C}$, we consider the conic bundle surface 
\[C'\colon\,\, X_0^2 + YX_0X_1 + YX_1^2 = tX_2^2\]
for $((X_0:X_1:X_2),Y)\in\PP^2_K\times\AAA^1_K$. 
Then, we obtain $\widetilde{C}$ by glueing together $C$ and $C'$ via $x_0 = X_0, x_1 = YX_1,x_2 = X_2, y = Y^{-1}$ on $\PP^2\times\PP^1\setminus\{\infty,0\}$.
	
\begin{lemma}\label{lemma_properties_proper_model} \hfill
	\begin{enumerate}
		\item The fibre of $\widetilde{\pi}$ over any point of $\AAA_K^1$ is smooth.
		\item The fibre of $\widetilde{\pi}$ over the point at infinity is integral but not geometrically reduced.	
		\item $\widetilde{C}$ is a regular proper model of $C$.
		\item $C$ admits no smooth proper model over $K$.
	\end{enumerate}
\end{lemma}
\begin{proof}
	Both (1) and (2) follow from a simple calculation. For (3), from (1) we know that $C$ is smooth over $\AAA_K^1$, hence smooth over $K$. Thus it suffices to consider the points 
	which lie over the fibre at infinity. The Jacobian matrix of $C'$ is
	\begin{equation*}
		(YX_1, YX_0, 0,X_1(X_0+X_1)).
	\end{equation*}
	Thus $C'$ is smooth everywhere apart from at the closed points
	$$P_1: Y = X_1 = 0, \quad X_0^2 = tX_2^2, \quad \quad P_2: Y=0, \quad X_0 = X_1, \quad X_0^2 = tX_2^2,$$
	of degree two, where it is not smooth. 
	From the equation it is clear that $X_1$ and $Y$ are a regular system of parameters
	for $P_1$ and that $X_0 + X_1$ and $Y$ are a regular system of parameters for $P_2$.
	Thus these points are regular, as required.
	
	For (4) assume that $C$ admits a smooth proper model $S$ over $K$.
	Up to performing birational transformations, we may assume that 
	$S$ is a relatively minimal
	conic bundle surface over $K$. By (1) this has at most one singular
	fibre, which lies over the point at infinity. Thus we obtain a Brauer--Severi scheme
	over $\mathbb{A}^1_K$, hence	a Brauer group element $\alpha \in \Br \mathbb{A}^1_K$.
	By \cite[Prop.~7, 8]{Coombes}, the conic bundle surface $S$ admits a section
	over the separable closure of $K$. It follows that the Brauer group element
	$\alpha$ is algebraic, i.e.~$\alpha \in \Br_1 \mathbb{A}^1_K$.
	However $\Br_1 \mathbb{A}^1_K = \Br K$ \cite[Prop.~5.4.2]{Brauer},
	thus $\alpha$ is constant, which is a contradiction (the fact that $\alpha$ is non-constant follows,
	for example, from Lemma \ref{lemma:measure_set_conic_with_points} below).
\end{proof}
	
By Lemma \ref{lemma_properties_proper_model} our proper model $\widetilde{C}$ is not smooth, and there is no smooth proper model. Hence, the framework of \cite{loughran_smeets} does not apply in this situation. Moreover, we cannot compute $\Delta(\widetilde{\pi})$ as the definition of $\delta_D(\widetilde{\pi})$ makes no sense for the fibre over the point at infinity. This is reduced of multiplicity $1$, however after the purely inseparable field extension $K(\sqrt{t})$ it becomes non-reduced. Given this, one is tempted to try to apply the framework from \cite{BLS}, regarding fibrations with multiple fibres. However \cite[Conj.~1.1]{BLS} predicts that the power of $B$ is smaller, which is not at all the case in our setting.
	
Non-reduced fibres cannot occur for smooth conic bundle surfaces: this follows for example from \cite[Prop.~7, 8]{Coombes} which implies that there is a section over the separable closure, and any section must meet each fibre in smooth point \cite[Cor.~10.1.9]{Brauer}. Nonetheless, for a smooth conic bundle surface with a non-split fibre over a single point $P$ we have $\delta_P(\widetilde{\pi}) = 1/2$. In our case we have a single non-split fibre. Therefore naively applying this formula in our situation we obtain $\Delta(\widetilde{\pi})=1/2$, which agrees with \Cref{main_theorem_rtl_points}.

However this factor of $1/2$ arises in a completely different way in our setting. Over number fields this comes from the fact that $1/2$ of all primes split in a separable quadratic extension. In our setting all places $\omega$ are relevant, but the measure of soluble fibres which reduce to the fibre at infinity moduli $\omega$ is $1/2$ the total measure. In fact calculating this measure is the most delicate part of the paper (this is Lemmas \ref{lemma:points_on_conic_omega_negative} and \ref{lemma:measure_set_conic_with_points}).
This difference can be interpreted as stemming from the fact over global fields of characteristic not equal to $2$, local solubility of conics can be encoded via the Hilbert symbol. However in characteristic $2$ a different symbol needs to be used which has very different formal properties. We use this symbol, which was introduced by Serre in \cite[Ch.~XIV \S5]{serre79}.  This symbol is closely related to the Brauer group element \eqref{eqn:alpha}.

The leading constant $c$ in Theorem \ref{main_theorem_rtl_points} can be interpreted via the conjectural constant given in \cite[Conj.~3.8]{loughran_rome_sofos}, despite as already explained the problem not lying in the exact framework. For example the factor $4$ should be interpreted as a cardinality of some subordinate Brauer group times a normalising factor from our choice of Haar measures. We have $\sqrt{\pi} = \Gamma(1/2)$. The factors $\left(1- 2^{-\deg \omega}\right)^{1/2}$ are convergence factors from $\zeta_K(s-1)^{1/2}$ and the local factors $c_\omega$ arise in our proof as special values of local Fourier transforms, thus can be interpreted as local densities.

\subsection{Structure of the paper}
	In \Cref{section:Tauberian_theorem}, we prove a Tauberian Theorem (\Cref{Tauberian_Theorem}) for a Dirichlet series that may have a branch point singularity.

	In Section \ref{sec:local_solubility} we study in detail local solubility of the conics in our family.
	With the use of class field theory, we can give an equivalence between existence of rational points and local symbols introduced by Serre (see \Cref{equivalence_rtl_points_local_symbol}). This 
	makes it possible to give a description of the $y$ giving $C_y(K_\omega)\neq\emptyset$ (Lemma \ref{lemma:points_on_conic_nonnegative_valuation}, 
	\ref{lemma:points_on_conic_place_t}, and \ref{lemma:points_on_conic_place_t_inv}).
	Obtaining a result for the fibres which reduce to the point at infinity modulo some $\omega\in\Omega_K$ is surprisingly indirect. 
	
	In \Cref{sec:harmonic_analysis} we compute and analyse the analytic behaviour of the height zeta function $Z(s)$ via harmonic analysis.
	We show that the Fourier transform vanishes for all non-trivial additive characters except for 
	a special character arising from Serre's 
	local symbols (\Cref{lemma:global_Fourier_tranform_vanishes_nontrivial_and_neq_infty}). It turns out that the Fourier transform at this special character is indistinguishable from the one at the trivial character on the set of interest, that is, on the image of the map $C(\Adele_K)\to \AAA^1_K(\Adele_K)$. This results in the fact that the height zeta function is given by twice the contribution from the Fourier transform at the trivial character. 

	We conclude the proof by applying the Tauberian \Cref{Tauberian_Theorem} to our height zeta function to obtain the asymptotic formula.

%
	
	
\subsection*{Acknowledgements}
JO was partially supported by the Caroline Herschel Programme of Leibniz University Hannover, as well as a scholarship for a research stay abroad of the Graduiertenakademie of Leibniz University Hannover at the University of Bath. DL was supported by UKRI Future Leaders Fellowship \texttt{MR/V021362/1}. We thank Philip Dittmann for help with the proof of \Cref{lemma:points_on_conic_omega_negative} and Abdulmuhsin Alfaraj for helpful comments. We thank Ofir Gorodetsky for useful correspondence on Tauberian theorems over function fields, and the referees for an extremely careful reading of the paper and useful comments.

\section{A Tauberian theorem}
\label{section:Tauberian_theorem}

\subsection{Statement}
We recall the basic properties of the zeta function of $\FF_q(t)$ from \cite[Thm.~5.9]{Rosen}.

\begin{lemma}\label{lemma:value_zeta_function}
	We have
	\[ \zeta_{\FF_q(t)}(s) := \prod_{\omega\in\Omega_{\FF_q(t)}} \left(1-q^{-s\deg \omega}\right)^{-1} 
	=\frac{1}{(1-q^{1-s})(1-q^{-s})}, \quad \Real(s) > 1. \]
	The zeta function $\zeta_{\FF_q(t)}(s)$ has a meromorphic continuation to $\Real(s) > 0$ 
	with simple poles at $s=1 + 2 \pi i n/\log q$ for $n \in \ZZ$.
\end{lemma}

Let $\varphi:\NN \to \CC$ be an arithmetic function. Consider the Dirichlet series
\begin{equation}\label{Dirichlet_series_F}
	F(s) = \sum_{M=0}^\infty \frac{\dirichlet(q^M)}{q^{Ms}}
\end{equation}
associated to $\dirichlet$ for $s\in\CC$, providing it exists.
Given the periodic nature of $F$, we work in the following region.
\begin{equation}\label{def_set_C}
 D= \left\{ s\in\CC: -\frac{\pi i}{\log q} \le \Image(s) <  \frac{\pi i}{\log q}\right\}.
\end{equation}
Given Lemma \ref{lemma:value_zeta_function}, our model for a branch singularity of  order $b$ at $s = a$ will be $(1 - q^{a-s})^{-b}$.

\begin{theorem}\label{Tauberian_Theorem}
	Let $a\in\RR$ be such that $F$ converges absolutely for $\Real(s) > a$. Assume that 
	there is $b> 0$ and a function $\widetilde{F}$ that is holomorphic in a neighbourhood of 
	the line
	$\lbrace s \in D : \Real(s) = a \rbrace$ such that
	\[F(s) = \frac{1}{\left(1-q^{a-s}\right)^b}\widetilde{F}(s), \quad \text{ for all }s \in D.\]
	Then
	\[\dirichlet(q^M) = \frac{\widetilde{F}(a)}{\Gamma(b)}q^{aM}M^{b-1} + 
	O\left(q^{aM}M^{b-2}\right), \quad \text{as } M \to \infty.\]
\end{theorem}

Numerous Tauberian theorems for power series with poles or branch point singularities can be found in the literature, which can be applied to Dirichlet series over function fields after a change of variables (e.g.~\cite[Thm.~11.3]{Odl95}). We struggled to find the exact statement in the literature we needed which moreover had a complete proof. As such for completeness we provide our own proof, which follows the classical method.

\subsection{Hankel contours}
Our approach is based upon \cite[Thm.~17.1]{Rosen} and \cite[Prop.~1.3]{datta23}. 
The following is \cite[Lemma 1]{porritt2020}, though we have a different sign  as we are taking a different orientation for the contour.
	
\begin{lemma} \label{lemma_Hankel_contour}
	Let $A,M,\delta>0$. Let $\mathcal{H}$ be the Hankel contour of radius 1 around 0 going along the negative real axis to $-M\delta$ in clockwise direction (see \Cref{fig:Hankel_countour_lemma}). Then, uniformly for $\absvalue{b}\le A$, we have
	\begin{equation*}
		\frac{1}{2\pi i}\int_{\mathcal{H}}w^{-b}\frac{\mathrm{d}w}{\left(1-\frac{w}{M}\right)^{M+1}} = -\frac{1}{\Gamma(b)} + O_{A,\delta}\left(\frac{1}{M}\right).
	\end{equation*}
\end{lemma}
	
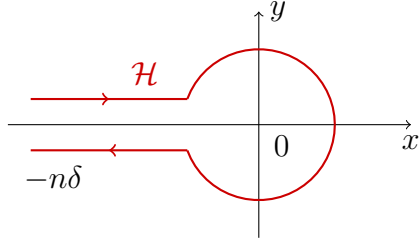
\begin{figure}[htb]
	\centering
	\begin{tikzpicture}
		\draw[->] (-3.3,0) -- (2,0);
		\draw[->] (0,-1.5) -- (0,1.5);
		\draw (2,0) node [below] {$x$};
		\draw (0,1.5) node [right] {$y$};
		\draw (0.3,0) node [below] {$0$};
		\draw (-2.7,-0.4) node [below] {$-M\delta$};
		\draw (-1.5,0.4) node [above, red!80!black] {$\mathcal{H}$};
		\draw [red!80!black, thick, domain=0:160] plot ({cos(\x)}, {sin(\x)});
		\draw [red!80!black, thick, domain=200:360] plot ({cos(\x)}, {sin(\x)});
		\draw [red!80!black, thick, decoration={
		markings,
			mark=at position 0.5 with {\arrow{>}}}, postaction={decorate}] (-3,0.34) -- (-0.94,0.34);
		\draw [red!80!black, thick, decoration={
				markings,
		mark=at position 0.5 with {\arrow{>}}}, postaction={decorate}] (-0.94,-0.34) -- (-3,-0.34);
	\end{tikzpicture}
	\caption{The Hankel contour $\mathcal{H}$ appearing in \Cref{lemma_Hankel_contour}}
\label{fig:Hankel_countour_lemma}
\end{figure}
	
By performing the variable change $u = q^{-s}$, the region $D$ transforms to $\CC^\times$ and the region $\Real(s) \geq a$ transforms to the punctured disk $\{ u \in \CC^\times : |u| \leq q^{-a}\}$. We obtain the power series
\begin{equation}
	F(u) = \sum_{M=0}^\infty \dirichlet(q^M)u^M,  \quad |u| < q^{-a}.
\end{equation}
A priori $F(u)$ is only holomorphic for $u \neq 0$, however holomorphicity at $u = 0$ is obtained from a simple application of the dominated convergence theorem and absolute convergence.

\begin{figure}[ht]
	\begin{tikzpicture}[scale=0.85]
		\draw[->] (-4.5,0) -- (4.5,0);
		\draw[->] (0,-4.5) -- (0,4.5);
		\draw [black, dashed] (0,0) circle (2);
		\draw [green!50!black, thick] (0,0) circle (1);
		\draw [->, black,thick,domain=2.5:110, samples=200,smooth] plot ({4*cos(\x)}, {4*sin(\x)});
		\draw [->, black,thick,domain=110:250, samples=200,smooth] plot ({4*cos(\x)}, {4*sin(\x)});
		\draw [black,thick,domain=250:357.5, samples=200,smooth] plot ({4*cos(\x)}, {4*sin(\x)});
		\draw [red!80!black, thick, domain=20:340] plot ({0.5*cos(\x)+2}, {0.5*sin(\x)});
		\draw [red!80!black, thick, decoration={
			markings,
			mark=at position 0.5 with {\arrow{>}}}, postaction={decorate}] (2.47,0.17) -- (4,0.17);
		\draw [red!80!black, thick, decoration={
			markings,
			mark=at position 0.5 with {\arrow{>}}}, postaction={decorate}] (4,-0.17) -- (2.47,-0.17);
		\draw (0,1.3) node [right, green!50!black] {$\lvert u\rvert = r $};				\draw (0,2.3) node [right, black] {$\lvert u\rvert = q^{-a}$};				\draw (0,4.3) node [right, black] {$\mathcal{C}: \lvert u\rvert = q^{-a\eta}$};
		\draw (3,0.3) node [red!80!black, above] {$\mathcal{H}'$};
		\draw (5,0) node [below] {$x$};
		\draw (0,5) node [right] {$y$};
	\end{tikzpicture}
	\caption{The contours in the proof of \Cref{main_theorem_rtl_points}}
	\label{fig:proof_thm}
\end{figure}
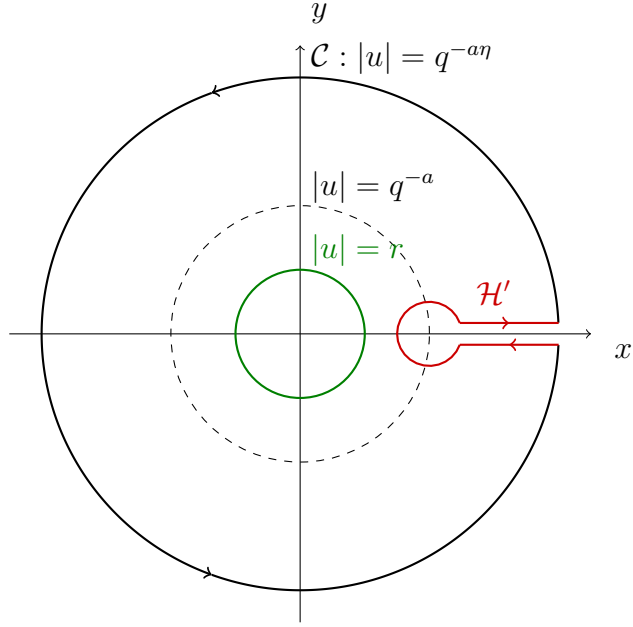

Let $C_r$ be a small circle around the origin of radius $r < q^{-a}$ (small green circle in \Cref{fig:proof_thm}), oriented anticlockwise. By applying Cauchy's integral formula, we obtain
\begin{equation}\label{N(B)_as_integral_1}
	\dirichlet(q^M) = \frac{1}{2\pi i}\int_{C_r}F(u)\frac{\mathrm{d}u}{u^{M+1}}.
\end{equation}
As $\{u\in \CC: \absvalue{u} \leq q^{-a}\}$ is compact and $\widetilde{F}$ is holomorphic on this region by assumption, there is $0 < \eta < 1 $ such that $F(u)$ is holomorphic on $\lbrace u\in \CC: \absvalue{u} \le q^{-a\eta}\rbrace$ except for a branch cut at $q^{-a}$. We take $\eta$ to be sufficiently close to $1$. Let $\mathcal{H}'$ be the contour as shown in red in \Cref{fig:proof_thm} that consists of a circle of radius $q^{-a}/M$ oriented in a clockwise direction around $q^{-a}$, with  two line segments parallel to the ray from 0 to $q^{-a}$ joining this small circle to the bigger black incomplete circle  $\mathcal{C}$ along $|u| = q^{-a\eta}$.

\noindent We apply Cauchy's integral formula to the  arrangement in \Cref{fig:proof_thm} to obtain 
\begin{equation}\label{N(B)_Cauchy_appliled}
	\dirichlet(q^M) = \frac{1}{2\pi i}\int_{\mathcal{H}'}F(u)\frac{\mathrm{d}u}{u^{M+1}} + \frac{1}{2\pi i} \int_{\mathcal{C}}F(u)\frac{\mathrm{d}u}{u^{M+1}} =: I_M + I_E.
\end{equation}

	
	
\subsection{Computation of the error and main term} 
\begin{lemma}\label{lemma_I_E}
	$I_E \ll_{q,\eta,b} q^{a\eta M}.$
\end{lemma}
\begin{proof}
	As $\widetilde{F}(u)$ is holomorphic in a compact neighbourhood of $\mathcal{C}$, it is 
	bounded in this domain. We obtain
	\begin{align*}
		|I_E| \leq \int_{\mathcal{C}}\frac{|\widetilde{F}(u)|\mathrm{d}u}{(1 - q^{a}u)^b|u|^{M+1}}
		\ll_{q,\eta,b} q^{a\eta (M+1)}
		\int_{\mathcal{C}}\frac{\mathrm{d}u}{|1 - q^{a}u|^b}.
	\end{align*}
	We now make the change of variables $u = q^{-a\eta} e^{2 \pi i \theta}$ to obtain
	$$|I_E| \ll_{q,\eta,b} q^{a\eta M} 
	\int	_{0}^{1} \frac{\mathrm{d}\theta}{|1 - q^{a(1-\eta)}e^{2 \pi i \theta}|^b}.$$
	The denominator is non-zero and the domain of integration is compact,
	hence this integral can be bounded in terms of $q,\eta$ and $b$.
\end{proof}
	
\begin{lemma}\label{lemma_I_M}
	The main term is given by
	\[I_M = 	\frac{\widetilde{F}(a)}{\Gamma(b)}q^{aM}M^{b-1} + O(q^{aM}M^{b-2}).\]
\end{lemma}
\begin{proof}
	As $\widetilde{F}(u)$ is analytic near $q^{-a}$, we can write $\widetilde{F}(u)$ as a Taylor series at $q^{-a}$
	\begin{equation*}
		\widetilde{F}(u) = \widetilde{F}(a) + a_1(u-q^{-a}) + a_2(u-q^{-a})^2 + \cdots,
	\end{equation*}
	where $a_i\ll1$ for all $i\ge 1$. Then, providing $\eta$ is taken sufficiently close to $1$, we get
	\begin{equation}\label{eq_I_M}
	I_M = \frac{1}{2\pi i}\widetilde{F}(a)\int_{\mathcal{H}'}\frac{1}{(1-q^{a}u)^{b}}\frac{\mathrm{d}u}{u^{M+1}} + O_\eta\left( \int_{\mathcal{H}'} \frac{u-q^{-a}}{(1-q^{a}u)^{b}}\frac{\mathrm{d}u}{u^{M+1}} \right).
	\end{equation}
	For the main term in \eqref{eq_I_M}, recall that the small red circle in $\mathcal{H}'$ has radius $q^{-a}/M$. Thus this can be parametrised by $u=q^{-a} + q^{-a}/M e^{-it}$ with $t\in[\varepsilon,2\pi - \varepsilon]$ and $\varepsilon>0$ small enough. Thus, we perform the change of variables 
	\[u = q^{-a}\left(1-\frac{w}{M}\right) \quad\text{or equivalently,}\quad w = M(1-q^au).\]
	This gives $\mathrm{d}u = -q^{-a}\mathrm{d}w/M$. Moreover for the horizontal rays,
	if $\absvalue{u}$ goes up to $q^{-a\eta}$, then $w$ goes down to $-M\delta$ with 
	$\delta = 1 - q^{a(1- \eta)}$.
	Thus the change of variables transforms $\mathcal{H}'$ into $\mathcal{H}$
	from Lemma \ref{lemma_Hankel_contour}. We obtain
	\begin{align*}
		\frac{\widetilde{F}(a)}{2\pi i}\int_{{\mathcal{H}}'}\frac{1}{(1-q^{a}u)^{b}}\frac{\mathrm{d}u}{u^{M+1}}
		&= \frac{\widetilde{F}(a)}{2\pi i}\int_{\mathcal{H}} \frac{1}{\left(\frac{w}{M}\right)^{b}} 
		\frac{-q^{-a}\mathrm{d}w/M}{q^{-a(M+1)}\left(1-\frac{w}{M}\right)^{M+1}}\\
		&= -\frac{\widetilde{F}(a)q^{aM}M^{b-1}}{2\pi i}
		\int_{\mathcal{H}} w^{-b} \frac{\mathrm{d}w}{\left(1-\frac{w}{M}\right)^{M+1}}\\
		&= \widetilde{F}(a)q^{aM}M^{b-1}\left(\frac{1}{\Gamma(b)} 
		+ O_\eta\left(\frac{1}{M}\right)\right),
	\end{align*}
	where the final step is \Cref{lemma_Hankel_contour}.
	For the error term in \eqref{eq_I_M}, we note that 
	$$\frac{u-q^{-a}}{(1-q^{a}u)^{b}} = \frac{-q^{-a}}{(1-q^{a}u)^{b-1}}.$$
	Thus by the above this contributes $O(q^{a(M)}M^{b-2}/q^{a})$ to $I_M$, as required.
\end{proof}
	
\Cref{Tauberian_Theorem} now follows immediately from \Cref{lemma_I_E,lemma_I_M}. \qed

\section{Local solubility} \label{sec:local_solubility}
\subsection{Preliminaries} \label{sec:preliminaries}
We recall some basic facts about the field $K= \FF_q(t)$.

Let $\ringofintegers = \FF_q[t]$ denote the ring of integers of $K$. Let $\Omega_K$ denote the set of places of $K$. Each place $\omega$ either corresponds to a monic, irreducible polynomial in $\ringofintegers$  or to the \emph{place at infinity} $t^{-1}$. We identify $\omega$ with the corresponding element of $K$. The completion $K_\omega$ of $K$ at $\omega$ can therefore be viewed as a power series field over $\FF_{q^{\deg \omega}}$, where by convention we take $\deg t^{-1} = 1$. We denote by $\mathcal{O}_\omega$ the ring of integers of $K_{\omega}$. We choose the unique Haar measure of $K_\omega$ such that $\mathcal{O}_\omega$ has volume $1$. For all $k \in \ZZ$ we have
\begin{equation} \label{eqn:mu_valuation}
		\mu_\omega\left(\{y\in K_\omega: v_\omega(y) = k\}\right) =  q^{-k\deg\omega}\left(1-q^{-\deg\omega}\right).
\end{equation}
Recall that over local fields of characteristic not equal to $2$, one can detect solubility using the Hilbert symbol. Serre has defined an analogous symbol in characteristic $2$, which we now recall from  \cite[XIV §5]{serre79}. From now on we take $K = \FF_2(t)$. Let $\Br K$ denote the Brauer group of $K$. We have the following exact sequence
\[0\longrightarrow\Br K\longrightarrow \bigoplus_\omega\Br K_\omega\overset{}{\longrightarrow} \QQ/\ZZ\longrightarrow 0,\]
where the last map is $(b_\omega)_\omega \mapsto  \sum_\omega \inv_{\omega} b_\omega$, due to \cite[p.~163, Example h)]{serre79}. Here $\inv_{\omega}\colon\Br K_\omega\rightarrow \QQ/\ZZ$ is the local invariant, which is an isomorphism.
	For $a\in K, b\in K^\times$, let $[a,b)\in\Br K$ be the local symbol defined in \cite[Ch.~XIV, \S5]{serre79}. We note that we can also replace $K$ by $K_\omega$ in this definition for any place $\omega$. For $a\in K_\omega,\  b\in K_\omega^\times$, we define \[[a,b)_\omega = 2\cdot\inv_{\omega}[a,b)\in\ZZ/2\ZZ.\] 
This is an element of $\ZZ/2\ZZ$, as the element $\inv_{\omega}[a,b)$ lies in $(1/2)\ZZ/\ZZ$.
By exactness of the above sequence, for $a\in K,b\in K^\times$ we obtain
\begin{equation}\label{eq:reciprocity_condition_on_sum_symbols}
	\sum_\omega [a,b)_\omega = 0.
\end{equation}
By \cite[Ch.~XIV, \S5 Prop.~14; Ch.~V, \S2 Prop.~3]{serre79} 
this has the following properties. 

\begin{lemma}\label{prop_Serre}
	Let $\omega$ be a place of $K$, and let $a,a'\in K_\omega$, $b,b'\in K_\omega^\times$. 
	\begin{enumerate}
		\item We have $[a,b)_\omega = 0$ if and only if $b$ is a norm in the Artin--Schreier extension $K_\omega(\alpha)/K_\omega$, where $\alpha^2-\alpha = a$,
		\item $[a+a',b)_\omega = [a,b)_\omega+[a',b)_\omega$,
		\item $[a,bb')_\omega = [a,b)_\omega+[a,b')_\omega$.
	\end{enumerate}
	Let $M_\omega$ be an unramified extension of $K_\omega$.
	\begin{enumerate}
		\item[(4)] The map $N_{M_\omega/K_\omega}:\calO_{M_\omega}^\times\to\calO_{K_\omega}^\times$ is surjective. Hence, units in $\calO_{K_\omega}$ are norms in the extension $M_\omega/K_\omega$.
	\end{enumerate}
\end{lemma}

In Lemma \ref{prop_Serre}, if the polynomial $\alpha ^2 - \alpha - a$ is reducible, then we take the convention of viewing the Artin--Schreier extension as a quadratic \'etale algebra.

\subsection{Application of local symbols}
We now prepare for the proof of \Cref{main_theorem_rtl_points}. Recall that our family is given by
\begin{equation} \label{eqn:conic_equation}
C\colon\,\, x_0^2 + x_0x_1 + yx_1^2 = tx_2^2\quad\subseteq \AAA^1_K\times\PP^2_{K}.
\end{equation}
Recall from Lemma \ref{lemma_properties_proper_model} that our family has a singular fibre 
over the point at infinity and no other singular fibres.
This singular fibre controls the arithmetic of the family.
We begin with an analysis of when a fibre has a $K_\omega$-point.		
We define the Artin--Schreier extension 
\begin{equation*}
	L = K(y)[a]/(a^2-a-y)\supseteq K(y).
\end{equation*}
The norm form of this extension is given by 
\begin{equation}\label{norm_form_Artin_Schreier}
		N_{L/K}(x_0+x_1a) 	= x_0^2-x_0x_1-x_1^2y,
\end{equation}
where $x_0,x_1 \in K$. We denote by $L_\omega = L \otimes_{K(y)} K_\omega(y)$ for a place $\omega$. 
	
\begin{prop}\label{equivalence_rtl_points_local_symbol}
	Let $\omega\in\Omega_K$ and $y\in K_\omega$. We have $C_y(K_\omega)\neq \emptyset$ if and only if $[y,t)_\omega = 0$, which is equivalent to $t$ being a norm from the Artin--Schreier extension $L_\omega$.
\end{prop}
\begin{proof}
	First note that if there is a solution with $x_2 = 0$ then there is a solution with $x_2 \neq 0$,
	since a smooth conic over an infinite field with a rational point has a Zariski dense set
	of rational points. The result is then immediate from Lemma \ref{prop_Serre} and the 
	equation of the conic \eqref{eqn:conic_equation},
	noting that $1 = -1 \in K$.
\end{proof}
	
\subsection{Good reduction}\label{Rational_Points_on_C}
We next consider the case of good reduction. The places $t$ and $t^{-1}$ behave differently than the other places. We deal with them separately later.
\begin{lemma}\label{lemma:points_on_conic_nonnegative_valuation}
Let $\omega$ be a place of $K$ with $\omega\neq t,t^{-1}$ and $y \in \calO_\omega$. Then the conic $C_y$ has a $K_\omega$-point.
\end{lemma}	
\begin{proof}
The Jacobian matrix of $C_y \bmod \omega$ is given by $(x_1, x_0, 0).$
	The only potential singular point is $x_0 = x_1 = 0$. This implies $tx_2^2 = 0$ on $C_y$ over $\FF_{\omega}$. As $t$ is a unit, this yields $x_2 = 0$. Hence, there is no singular point on $C_y$ over $\FF_{\omega}$. So $C_y$ has good reduction over $K_\omega$. The Chevalley--Warning Theorem now implies that there is a smooth $\FF_\omega$-point, which thus lifts to a $K_\omega$-point by Hensel's Lemma.
\end{proof}

\subsection{Bad reduction}
We are now interested in the fibres $\pi^{-1}(y)$, which reduce modulo $\omega$ to the point at infinity, that is $y$ with negative $\omega$-adic valuation. We begin with the following symbol computation.

\begin{lemma}\label{lemma:points_on_conic_omega_negative}
	Let $\omega\in\Omega_K$ with $\omega\neq t,t^{-1}$ and $n\in\NN$. We view
	$\omega$ as a monic irreducible polynomial. Then 
	$[\omega^{-n},t)_\omega = 1.$
\end{lemma}
\begin{proof}
	A direct computation of this symbol seems to be hard, so we calculate all symbols $[\omega^{-n},t)_v$ for $v\neq \omega$ and then use (\ref{eq:reciprocity_condition_on_sum_symbols}).
	
	Let $v$ be a place of $K$. The relevant Artin-Schreier extension is $L_{n,v} := K_v[x]/(x^2 - x - \omega^{-n})$. First assume that $v\neq \omega,t,t^{-1}$. Then $L_{n,v}$ is unramified at $v$ and $t$ is a unit, hence $[\omega^{-n},t)_v = 0$ by \Cref{prop_Serre}.
	Next, consider the place $t$. Here we shall prove that
	\begin{equation} \label{eqn:symbol_t}
	[\omega^{-n},t)_t = 1.
	\end{equation}
	Assume for a contradiction that $C_{\omega^{-n}}(K_t)\neq\emptyset$. Consequently, there are $x_0,x_1,x_2\in\ringofintegerst$ not all divisible by $t$ such that $x_0^2+x_0x_1+\omega^{-n}x_1^2 = tx_2^2$. 
	By substituting $x_1$ by $\omega^n x_1$, we get 
	\begin{equation}\label{eq:case_t_proof_rational_points_for_negative_valuation_y}
		x_0^2+\omega^nx_0x_1 + \omega^nx_1^2 = tx_2^2.
	\end{equation}
	As $\omega\neq t$ is an irreducible polynomial in $\FF_2[t]$, we can write \[\omega = a_0+a_1t+a_2t^2+\dots +t^m\] for an $m\in\NN$, $a_i\in\FF_2$ with $a_0 = 1$. It follows, $\omega\equiv\omega^n \equiv 1\bmod t$. Considering (\ref{eq:case_t_proof_rational_points_for_negative_valuation_y}) modulo $t$ gives
	\[x_0^2 + x_0x_1 + x_1^2 \equiv 0\bmod t.\]
	This equation has only the trivial solution $x_0\equiv x_1\equiv 0\bmod t$. But then
	\[x_2^2t = t^2(x_0^2+\omega^nx_0x_1+\omega^nx_1^2)\]
	and $x_0^2+\omega^nx_0x_1+\omega^nx_1^2\in \FF_2[t]$, which implies $t \mid x_2$. This contradicts that $t$ does not divide every $x_i$.  Thus, $C_{\omega^{-n}}(K_t) = \emptyset$ and \Cref{equivalence_rtl_points_local_symbol} yields $[\omega^{-n},t)_t\neq 0$, hence $[\omega^{-n},t)_t = 1$.
			
	Lastly, consider the infinite place $t^{-1}$. In equation \eqref{eqn:conic_equation}
	we make the change of variables $x_2 \mapsto t^{-1}x_2$ to obtain
	\begin{equation}\label{eq:case_t-1_proof_rational_points_for_negative_valuation_y}
		x_0^2 + x_0x_1 + \omega^{-n}x_1^2 = t^{-1}x_2^2.
	\end{equation}
	As $\omega$ is a monic polynomial of positive degree we have $v_{t^{-1}}(\omega^{-n}) > 0$.
	We consider (\ref{eq:case_t-1_proof_rational_points_for_negative_valuation_y}) modulo $t^{-1}$ and obtain
	\[x_0^2 + x_0x_1 \equiv 0\bmod t^{-1}.\]
	This equation has the non-singular solution $x_0\equiv x_1\equiv 1\bmod t^{-1}$ that we can lift to $\ringofintegerstinv$ by Hensel's Lemma. Hence $C_{\omega^{-n}}(K_{t^{-1}})\neq \emptyset$ and \Cref{equivalence_rtl_points_local_symbol} 
	yields $[\omega^{-n},t)_{t^{-1}} = 0$.
			
	By combining these results and using \eqref{eq:reciprocity_condition_on_sum_symbols}, we get 
	$[\omega^{-n},t)_\omega = -[\omega^{-n},t)_t = 1$.
\end{proof}

We use this lemma to determine the measure of the set of conics with local point.

			
\begin{lemma}\label{lemma:measure_set_conic_with_points}
	Let $\omega$ be a place of $K$ with $\omega\neq t,t^{-1}$. For $k\in\NN$, consider the two sets
	\begin{align*}
	\mathcal{A}_{0,k} &= \left\{y\in K_\omega: v_\omega(y) = -k, ~C_{y}(K_\omega)\neq \emptyset\right\},  \\
		\mathcal{A}_{1,k} &= \left\{y\in K_\omega: v_\omega(y) = -k, ~C_{y}(K_\omega) = \emptyset\right\}.
	\end{align*}
	Let $\mu_\omega$ denote the Haar measure on $K_\omega$. For all $k>1$, we have
	\begin{equation*}
		\mu_\omega\left(\mathcal{A}_{0,k}\right) = \mu_\omega\left(\mathcal{A}_{1,k}\right) = \tfrac{1}{2}\cdot\mu_\omega\left(\left\{y\in K_\omega: v_\omega(y) = -k\right\}\right)
		=  \tfrac{1}{2}\cdot2^{k\deg\omega}\left(1-2^{-\deg\omega}\right).
	\end{equation*}
	For $k=1$, we have
	\[\mu_\omega(\mathcal{A}_{0,1}) = 2^{\deg\omega-1} -1.\]
\end{lemma}	
\begin{proof}
	Let $y\in K_\omega$ with $v_\omega(y_\omega) < 0$. Write $y = u\omega^{-k}$ for 
	$u\in \ringofintegersw^\times$. Bilinearity of the local symbol from \Cref{prop_Serre} together with \Cref{lemma:points_on_conic_omega_negative} gives
	\begin{equation}\label{eq_g_switches}
	[y + \omega^{-1},t)_\omega = [y,t)_\omega + [\omega^{-1},t)_\omega = [y,t)_\omega + 1.
	\end{equation}
	Define the map
	\begin{align*}
		g:  K_\omega \rightarrow  K_\omega, \quad y \mapsto y+\omega^{-1}.
	\end{align*}
	Clearly, $g$ is injective. Further, it is measure preserving since the Haar measure is translation invariant. Moreover, as  $[y,t)_\omega$ is either 1 or 0, formula \eqref{eq_g_switches} implies that the map $g$ switches the elements whose conic has a point and those who do not. It follows immediately
	that $\mu_\omega(\mathcal{A}_{0,k}) = \mu_\omega(\mathcal{A}_{1,k})$ for $k > 1$, since $g$
	also preserves valuations in these sets. The calculation of the measure for $k> 1$
	then follows from \eqref{eqn:mu_valuation}.
	When $k = 1$, define for $i\in\lbrace 0,1\rbrace$ the sets
	\begin{equation*}
		\Omega_{i} = \left\{y\in\mathcal{A}_{i,1}: v_\omega(y+\omega^{-1}) \neq -1\right\}
		= \left\{y\in\mathcal{A}_{i,1}: y  \in \omega^{-1} + \ringofintegersw\right\}.
	\end{equation*}
	We have $\Omega_0 = \emptyset$; indeed for $x \in \ringofintegersw$ we have
	\[[\omega^{-1} + x,t)_\omega = [\omega^{-1},t)_\omega + [x,t)_\omega = 1,\] due to \Cref{lemma:points_on_conic_nonnegative_valuation,lemma:points_on_conic_omega_negative}, so the emptiness follows from \Cref{equivalence_rtl_points_local_symbol}.
	The same argument yields $\Omega_1 = \omega^{-1} + \ringofintegersw$.
	We obtain
	\[g(\mathcal{A}_{1,1}\setminus\Omega_1) = \mathcal{A}_{0,1}\setminus\Omega_0 = \mathcal{A}_{0,1},\]
	which implies
	\[\mu_\omega(\mathcal{A}_{0,1}) = \mu_\omega(\mathcal{A}_{1,1}) - \mu_\omega(\Omega_1)
	= \mu_\omega(\mathcal{A}_{1,1}) - 1.\]
	By \eqref{eqn:mu_valuation} we have
	\[\mu_\omega(\mathcal{A}_{1,1}) + \mu_\omega(\mathcal{A}_{0,1}) = \mu_\omega(\{y\in K_\omega: v_\omega(y) = -1\}) =  2^{\deg\omega} - 1.\]
	Thus taking these together we deduce that
	$$\mu_\omega(\mathcal{A}_{0,1}) = -\mu_\omega(\mathcal{A}_{0,1}) + 2^{\deg \omega} - 1
	= -\mu_\omega(\mathcal{A}_{0,1}) + 2^{\deg \omega} -2.$$
	Rearranging gives the statement of the lemma.
\end{proof}

\subsection{The places $t$ and $t^{-1}$}
It remains to determine local solubility over $K_t$ and $K_{t^{-1}}$.
				
\begin{lemma}\label{lemma:points_on_conic_place_t}
Write $y\in K_t$ as $y = \sum_{i\ge N}a_it^i$ for $N\in\ZZ$ and coefficients $a_i\in\FF_2$. We set $a_i = 0$ for all $i < N$. Then, $C_y$ has a $K_t$-point if and only if $a_0 = 0$.
\end{lemma}
\begin{proof}
	By the residue formula \cite[Ch.~XIV §5 Cor.~to Prop.~15]{serre79}, we get
	\[ [y,t)_t = \Res_t\left( y\frac{\mathrm{d}t}{t}\right) = a_0. \]
	Combining this with \Cref{equivalence_rtl_points_local_symbol}, the lemma follows.
\end{proof}
				
\begin{lemma}\label{lemma:points_on_conic_place_t_inv}
	Write $y\in K_{t^{-1}}$ as $y = \sum_{i\le N}a_it^i$ for some $N\in\ZZ$ and coefficients $a_i\in\FF_2$. Set $a_i = 0$ for all $i > N$. Then, $C_y$ has a $K_{t^{-1}}$-point if and only if $a_0 = 0$.
\end{lemma}
\begin{proof}
	Analogous to Lemma \ref{lemma:points_on_conic_place_t}.
\end{proof}
				

\section{Harmonic analysis} \label{sec:harmonic_analysis}
\subsection{Poisson summation}
We now begin in earnest to analyse the height zeta function \eqref{def_height_zeta_function} using Fourier analysis (we keep $K = \FF_2(t)$). The volume of $\ringofintegersw$ with respect to the Haar measure $\mathrm{d}x_\omega$ is equal to 1 for all $\omega\in\Omega_K$. We have an induced Haar measure $\mathrm{d}x$ on the ring of adeles $\Adele_K$ of $K$ which satisfies $\vol(\Adele_K/K) = q^{g-1} = 1/2$ by \cite[Cor.~to Thm.~VI.1]{Wei95}. Let $\psi:\Adele_K \to \CC^\times$ be an additive adelic (unitary) character. Such a character is uniquely determined by a collection of local characters $\psi_\omega: K_\omega \to \CC^{\times}$ that are trivial on $\O_\omega$ for all but finitely many $\omega$. We call $\psi$ \emph{automorphic} if it is trivial on the image of $K$.

We have the local heights $H_\omega(y_\omega) = \max\lbrace 1,\absvalue{y_\omega}_\omega \rbrace$; these give rise to an adelic height $H:= \prod_\omega H_\omega: \Adele_K \to \QQ$. For $s\in\CC$, we denote by $H(s;\cdot):= H(\cdot)^s$ and let 
\begin{align*}
	\widehat{H}(s;\psi) = \int_{\Adele_K}H(s;x)^{-1}\psi(x)\mathrm{d}x, \quad \Real(s) \gg 1, 
\end{align*}
be the Fourier transform of $H(s;\cdot)^{-1}$ on $\Adele_K$ at the character $\psi$.

As $H$ is a product of local heights, we have a product
\[\widehat{H}(s;\psi) = \prod_{\omega\in\Omega_K}\widehat{H}_\omega(s;\psi_{\omega}) = \prod_{\omega\in\Omega_K} \int_{K_\omega}H_\omega(s;x_\omega)^{-1}\psi(x_\omega)\mathrm{d}x_\omega.\]
of local Fourier transforms. We define the indicator function $f\colon K\rightarrow\{0,1\}$ by 
\begin{equation} \label{def:f}
f(y) = \begin{cases}
	0, &\text{if}~C_y(K) = \emptyset, \\
	1, &\text{if}~C_y(K) \neq \emptyset.
\end{cases}
\end{equation}
This allows us to rewrite the height zeta function as
\[Z(s) = \sum_{y\in K}f(y)H(s;y)^{-1}.\]
The Hasse principle \cite[Ch.~VI, 3.1]{lam05} shows that we can write $f = \prod_{\omega\in\Omega_K}f_\omega$ with $f_\omega$ defined analogously on $K_\omega$ . 
We analogously define the Fourier transform of $f(\cdot)H(s;\cdot)^{-1}$ on $\Adele_K$ at the character $\psi$ by
\begin{equation} \label{def:Fourier_transform}
\widehat{fH}(s;\psi) = \int_{\Adele_K}f(x)H(s;x)^{-1}\psi(x)\mathrm{d}x, \quad \Real(s) \gg 1.
\end{equation}
Again, we can write $\widehat{fH}(s;\psi)$ as a product of local Fourier transforms
\[\widehat{fH}(s;\psi) = \prod_{\omega\in\Omega_K} \widehat{f_\omega H_\omega}(s;\psi_{\omega}) = \prod_{\omega\in\Omega_K}\int_{K_\omega} f(x_\omega)H(s;x_\omega)^{-1}\psi_{\omega}(x_\omega)\mathrm{d}x_\omega,\]
as $f$ can be written as a product over all $f_\omega$. By applying Poisson summation as stated in \cite[Thm.~3.35]{bourqui_zeta11} plus \cite[Cor.~3.36]{bourqui_zeta11} with $\Omega = \prod_{\omega\in\Omega_K} \calO_{K_\omega}$, we can rewrite the height zeta function as follows
\begin{equation} \label{eq:rewritten_height_zeta}
	Z(s)= 2 \sum_{\psi} \widehat{fH}(s;\psi), \quad \Real(s) \gg 1,
\end{equation}
where the sum is over all additive automorphic characters of $K$ and we use 
$1/\vol(\Adele_K/K) = 2$.

\subsection{Fourier transforms at the trivial character}\label{subsection:Fourier_transform_trivial_character}
For completeness we start with the local Fourier transform of $H$ at the trivial character $\one$. 
				
\begin{lemma}\label{lemma:local_fourier_transform_trivial_character}
	For every place $\omega\in\Omega_K$  we have
	\[\widehat{H_\omega}(s,\one) = \frac{1-2^{-s\deg \omega}}{1-2^{-(s-1)\deg \omega}}, \quad
	\Real(s) > 1.\]
	In particular $\widehat{H}(s,\one) = \zeta_K(s-1)\zeta_K(s)^{-1}$ for $\Real(s) > 2$.
\end{lemma}
\begin{proof}
	Since $\one$ and the height function $H$ are $1$ on $\ringofintegersw$, we obtain 
	\begin{equation*}
		\widehat{H_\omega}(s,\one) = \int_{y_\omega\in\ringofintegersw}\mathrm{d}y_\omega + \int_{\substack{y_\omega\in K_\omega\\v_\omega(y) < 0}} H(s;y_\omega)^{-1}\mathrm{d}y_\omega
		= 1 + \sum_{k=1}^{\infty} 2^{-sk\deg \omega}\int_{\substack{y_\omega\in K_\omega\\v_\omega(y) = -k}}\mathrm{d}y_\omega.
	\end{equation*}
	By \eqref{eqn:mu_valuation} this infinite sum is
	\begin{align*}
		\left(1 - 2^{-\deg \omega}\right)\sum_{k=1}^{\infty} 2^{-k(s-1)\deg \omega}
		= \frac{1 - 2^{-\deg \omega}}{2^{(s-1)\deg \omega}-1}.
	\end{align*}
	Combining we obtain
	\begin{equation*}
		\widehat{H_\omega}(s,\one) 
		= \frac{2^{(s-1)\deg \omega} - 2^{-\deg \omega}}{2^{(s-1)\deg \omega}-1} 
		= \frac{1-2^{-s\deg \omega}}{1-2^{-(s-1)\deg \omega}}.
	\end{equation*}
	The last part follows from Lemma \ref{lemma:value_zeta_function}.
\end{proof}

\begin{lemma}\label{lemma:fourier_transform_trivial_f_omega}
	Let $\omega\in\Omega_K\setminus\{t,t^{-1}\}$. We have
	\[\widehat{f_\omega H_\omega}(s;\one) = 
		1 + \frac{1}{2}
		\left(\frac{1}{2^{(s-1)\deg\omega}}\left(1 + 
		\frac{1-2^{-\deg\omega}}{2^{(s-1)\deg \omega}-1}\right) 
		- \frac{1}{2^{s\deg\omega -1}} \right), \quad \Real(s) > 1.\]
	\[\widehat{f_\omega H_\omega}(s;\one) = 
		1 + \frac{1}{2}\cdot \frac{1}{2^{(s-1)\deg\omega}}
		\left(1 - \frac{1}{2^{\deg\omega -1}} + 
		\frac{1-2^{-\deg\omega}}{2^{(s-1)\deg \omega}(1 - 2^{-(s-1)\deg \omega})}\right), 
		\quad \Real(s) > 1.\]		
\end{lemma}
\begin{proof}
	Applying \Cref{lemma:points_on_conic_nonnegative_valuation,lemma:measure_set_conic_with_points}
	and using a similar strategy to the proof of 
	Lemma~\ref{lemma:local_fourier_transform_trivial_character}, we obtain
	\begin{equation*}
		\widehat{f_\omega H_\omega}(s;\one) = 
		1 + 2^{-s\deg \omega}(2^{\deg\omega-1} -1) + 
		\frac{1}{2}
		\sum_{k=2}^{\infty} 2^{-sk\deg \omega}\cdot2^{k\deg\omega}\left(1-2^{-\deg\omega}\right).
	\end{equation*}	
	This infinite sum equals
	$$\frac{\left(1-2^{-\deg\omega}\right)}{2}
		\sum_{k=2}^{\infty} 2^{-k(s-1)\deg \omega}
	=\frac{\left(1-2^{-\deg\omega}\right)}{2} \cdot 
	\frac{1}{2^{(s-1)\deg \omega}(2^{(s-1)\deg \omega}-1)}.$$
	Combining these gives the statement.
\end{proof}

\begin{lemma}\label{lemma:fourier_transform_trivial_f_t_and_t_inv}
	Let $\omega = t$ or $\omega = t^{-1}$. Then
	\[\widehat{f_\omega H_\omega}(s;\one) 
	= \frac{1}{2}\cdot \frac{1-2^{-s}}{1-2^{(1-s)}}, \quad \Real(s) > 1.\]
\end{lemma}
\begin{proof}
	The map $y \mapsto y + 1$ is both measure and height preserving.
	Moreover by Lemmas~\ref{lemma:points_on_conic_place_t} and \ref{lemma:points_on_conic_place_t_inv}
	we have $f_\omega(y) + f_{\omega}(y + 1) = 1$. Thus
	$\widehat{f_\omega H_\omega}(s;\one) = \frac{1}{2} \widehat{H}_\omega(s;\one)$, and so the result
	follows from Lemma \ref{lemma:local_fourier_transform_trivial_character}.
\end{proof}
				
\begin{lemma}\label{Global_Fourier_Transform_at_trivial_character}
	The Fourier transform $\widehat{fH}(s;\one)$ equals
	\begin{align*}
		 \left(\frac{1}{2} \left( \frac{1-2^{-s}}{1-2^{1-s}} \right)\right)^2
		\prod_{\omega\in\Omega_K\setminus\{t,t^{-1}\}}
		\left(	1 +  \frac{1}{2}
		\left(\frac{1}{2^{(s-1)\deg\omega}}\left(1 + 
		\frac{1-2^{-\deg\omega}}{2^{(s-1)\deg \omega}-1}\right) 
		- \frac{1}{2^{s\deg\omega -1}} \right)\right)
	\end{align*}
	with the Euler product being absolutely convergent for $\Real(s) > 2$.
	
	Moreover $\widehat{fH}(s;\one)\zeta_K(s-1)^{-\frac 12}$ is given
	by an Euler product which is absolutely convergent and non-zero on $\Real(s) > 3/2$.
\end{lemma}
\begin{proof}
	The calculation of the Fourier transform follows from
	Lemmas \ref{lemma:fourier_transform_trivial_f_omega} and
	\ref{lemma:fourier_transform_trivial_f_t_and_t_inv}.	
	
	For the second part, for any place $\omega \neq t,t^{-1}$ we rewrite the local Fourier transform as 
    	\begin{equation} \label{eqn:rewrite}
    	\widehat{f_\omega H_\omega}(s;\one) = 
		1 + \frac{1}{2}\cdot \frac{1}{2^{(s-1)\deg\omega}}
		\left(1 - \frac{1}{2^{\deg\omega -1}} + 
		\frac{1-2^{-\deg\omega}}{2^{(s-1)\deg \omega}(1 - 2^{-(s-1)\deg \omega})}\right).
	\end{equation}
	We note that for $\Real(s) > 3/2$ we have
	\begin{equation} \label{eqn:bound_trivial}
	\left|\frac{1-2^{-\deg\omega}}{1 - 2^{-(s-1)\deg \omega}}\right|
	<1+2^{-\deg\omega/2}.
	\end{equation}
	For $\Real(s) > 3/2$ we obtain
	$$\widehat{f_\omega H_\omega}(s;\one) = 1 + \frac{1}{2} \cdot \frac{1}{2^{(s-1)\deg \omega}}
		+ O\left(\frac{1}{2^{(2(\Real s-1) \deg \omega}} + \frac{1}{2^{\Real(s) \deg \omega} }\right).$$
	We deduce that
	$$\left(1-2^{-s\deg \omega}\right)
	\widehat{f_\omega H_\omega}(s;\one)^2 = 1 + 
	 O\left(\frac{1}{2^{2\Real(s-1) \deg \omega}} + \frac{1}{2^{\Real(s) \deg \omega} }\right).$$
	Recalling the description of $\zeta_K(s)$ from 
	Lemma \ref{lemma:value_zeta_function}, we find that
	$\widehat{fH}(s;\one)^2\zeta_K(s-1)$ is given
	by an absolutely convergent Euler product on $\Real(s) > 3/2$.
	To complete the proof it suffices to show that it is non-zero in this domain,
	since this allows us to take square roots. 
	From absolute convergence, it suffices to show that 
    $\widehat{f_\omega H_\omega}(s;\one)$ is non-zero for $\Real(s) > 3/2$.
    For $\omega \in {t,t^{-1}}$ this is clear. From \eqref{eqn:bound_trivial},
    for $\Real s > 3/2$ we find that
	\begin{align*}
		&\left|\frac{1}{2}\cdot \frac{1}{2^{(s-1)\deg\omega}}
		\left(1 - \frac{1}{2^{\deg\omega -1}} + 
		\frac{1-2^{-\deg\omega}}{2^{(s-1)\deg \omega}(1 - 2^{-(s-1)\deg \omega})}\right)\right| \\
		&< \frac{1}{2}\cdot \frac{1}{2^{\deg\omega/2}}
		\left(1 - \frac{1}{2^{\deg\omega -1}} + 1 + \frac{1}{2^{\deg \omega/2}}\right)\\
		& < \frac{1}{2}\cdot \frac{1}{2^{1/2}}
		\left(2 + \frac{1}{2^{1/2}}\right) < 1,
	\end{align*}
	from which non-vanishing easily follows from \eqref{eqn:rewrite}.
\end{proof}

\subsection{Fourier transforms at the non-trivial characters}
In this subsection, we compute the Fourier transforms at the non-trivial characters. A special role will be played by the character
\begin{equation} \label{def:psi_t}
	\psi_{1/t}: \Adele_K \to \CC^\times, \quad (y_\omega) \mapsto \prod_\omega \exp(\pi i [y_\omega,t)_\omega),	
\end{equation}
where the local symbol is as in Section \ref{sec:preliminaries}. This is an additive adelic character by Lemma~\ref{prop_Serre}, and automorphic by \eqref{eq:reciprocity_condition_on_sum_symbols}. It is non-trivial by Lemma \ref{lemma:points_on_conic_omega_negative}.
We prove that the global Fourier transform vanishes for every non-trivial automorphic character, except $\psi_{1/t}$ where the Fourier transform equals the Fourier transform of the trivial character. 
		
\begin{lemma}\label{lemma_global_Fourier_vanishes_nontrivial_a_neq_t_inv}
Consider the compact subgroup \[\calO_C = \prod_{\omega_\in\Omega_K\setminus\{t,t^{-1}\}}\ringofintegersw\times t\ringofintegerst\times t^{-1}\ringofintegerstinv \quad \subseteq \Adele_K.\]
The height function $H$ and the indicator function $f$ are invariant under the additive action of $\calO_C$. The only additive adelic automorphic characters that are trivial on $\O_{C}$ are the trivial character $\one$ and $\psi_{1/t}$.
\end{lemma}
\begin{proof}
	It is easy to see from the definition of the height function that $H_\omega$ is invariant under $\prod_\omega \O_\omega$, hence under $\O_C$.	\Cref{lemma:points_on_conic_nonnegative_valuation,lemma:points_on_conic_place_t,lemma:points_on_conic_place_t_inv} show that $f$ is 1 on $\calO_C$. Then, \Cref{equivalence_rtl_points_local_symbol} together with \Cref{prop_Serre}\,(2) yields 
\[[y_\omega+x_\omega,t)_\omega = [y_\omega,t)_\omega + [x_\omega,t)_\omega = [y_\omega,t)_\omega\]
for all $\omega\in\Omega_K$, $y = (y_\omega)_\omega\in \Adele_K$ and $x = (x_\omega)_\omega\in\calO_C$. Hence, by using again \Cref{equivalence_rtl_points_local_symbol}, we get
\[f(y+x) = \prod_{\omega\in\Omega_K}f_\omega(y_\omega+x_\omega) = \prod_{\omega\in\Omega_K}f_\omega(y_\omega) = f(y).\]
Thus $f$ is invariant under the action of $\calO_C$.

Recall \cite[Lem.~IV.2.4]{Wei95} that $K$ and $\prod_{\omega} \O_\omega$ together generate $\Adele_K$. As $\O_C \subseteq \prod_{\omega} \O_\omega$ has index $4$, it follows that the subgroup of $\Adele_K$ generated by $K$ and $\O_C$ has index at most $4$ in $\Adele_K$. But $1$ has trivial image in the quotient,
whence one sees that the index is actually at most $2$. However $\one$ and $\psi_{1/t}$ are both trivial on $\O_C$, by 	\Cref{lemma:points_on_conic_nonnegative_valuation,lemma:points_on_conic_place_t,lemma:points_on_conic_place_t_inv}. The result now easily follows.
\end{proof}
				
\begin{lemma}\label{lemma:global_Fourier_tranform_vanishes_nontrivial_and_neq_infty}
	Let $\psi$ be an additive adelic automorphic character. Then 
	$$
	\widehat{fH}(s,\psi)  = 
	\begin{cases}
		\widehat{fH}(s,\one), \quad & \text{if }\psi \in \{\one, \psi_{1/t}\}, \\
		0, \quad & \text{otherwise.}
	\end{cases}
	$$
\end{lemma}
\begin{proof}
	First assume that $\psi$ is non-trivial on $\O_C$. 
	Due to \Cref{lemma_global_Fourier_vanishes_nontrivial_a_neq_t_inv} we know that
	$H$ and $f$ are invariant under $\calO_C$. It follows that
	\begin{align*}
		\widehat{fH}(s,\psi) 
		&= \int_{\Adele_K} H(y)^{-s}f(y)\psi(y)\mathrm{d}y\\
		&= \sum_{y\in\Adele_K/\calO_C} \int_{\calO_C}H(y+x)^{-s}f(y+x)\psi(y+x)\mathrm{d}x\\
		&= \sum_{y\in\Adele_K/\calO_C} H(y)^{-s}f(y)\psi(y) \int_{\calO_C} \psi(x)\mathrm{d}x.
	\end{align*}
	However the integral of $\psi$ over $\O_c$ equals $0$ by character orthogonality,
	as $\O_C$ is compact and $\psi$ is non-trivial on $\O_C$.

	By \Cref{lemma_global_Fourier_vanishes_nontrivial_a_neq_t_inv}  it suffices
	to consider the case $\psi = \psi_{1/t}$. However by \eqref{def:psi_t}, Proposition~\ref{equivalence_rtl_points_local_symbol} and the
	definition \eqref{def:f} of $f$, we have $f(y)\psi_{1/t}(y) = f(y)$ for all
	$y \in \Adele_K$. It then follows from the definition \eqref{def:Fourier_transform}
	that $\widehat{fH}(s,\psi_{1/t}) = \widehat{fH}(s,\one)$.
\end{proof}

\subsection{The height zeta function}
We now explain the consequences of the above analysis for the height zeta function \eqref{def_height_zeta_function}. In particular
we show that the height zeta function has an Euler product expansion, which is  not obvious from the definition. 
				
\begin{theorem}\label{thm:height_zeta_function}
	We have $Z(s) = 4\cdot \widehat{fH}(s;\one)$.
	Moreover $Z(s)\zeta_K(s-1)^{-\frac 12}$ is given
	by an Euler product which is absolutely convergent and non-zero on $\Real(s) > 3/2$.
	Hence $Z(s)$ admits a holomorphic continuation to the region
	$\Real(s) > 3/2$ except for branch points of order $1/2$ at $s = 2 + 2 \pi i n/\log 2$
	for $n \in \ZZ$.
\end{theorem}
\begin{proof}
	The first equality follows by Poisson summation  \eqref{eq:rewritten_height_zeta} and 
	Lemma \ref{lemma:global_Fourier_tranform_vanishes_nontrivial_and_neq_infty}.
	The rest is Lemma \ref{Global_Fourier_Transform_at_trivial_character}.
\end{proof}

\subsection{Proof of \Cref{main_theorem_rtl_points}}\label{section:Proof_of_main_theorem}
Note that
\[Z(s) = \sum_{M = 0}^\infty \frac{N(\AAA^1_K,\pi,2^M)}{q^{Ms}}\]
using the notation from \eqref{def:N(B)}, so we are exactly in the setting of \eqref{Dirichlet_series_F}.
In the light of the analytic properties proved in Theorem \ref{thm:height_zeta_function}, we can use  Theorem \ref{Tauberian_Theorem}  to give an asymptotic formula for $N(\AAA^1_K,\pi,2^M)$.

Via Lemma \ref{lemma:value_zeta_function} we can write $\widetilde{Z}(s) := Z(s)(1 - 2^{2-s})^{1/2} = Z(s)\zeta_K(s-1)^{-1/2}(1 - 2^{1-s})^{-1/2}$. Thus applying Theorem \ref{Tauberian_Theorem} with $a = 2$, $b = 1/2$, and $B = 2^M$ yields
$$
N(\AAA^1_K,\pi,B) = \frac{cB^2}{(\log B)^{1/2}} + O\left(\frac{B^2}{(\log B)^{3/2}}\right),$$
where
$$c =\frac{4(\log 2)^{1/2}}{\Gamma(1/2)(1 - 2^{-1})^{1/2}}\prod_{\omega}\left(1-\frac{1}{2^{\deg \omega}}\right)^{1/2} c_\omega,$$
and $c_\omega = \widehat{f_\omega H_\omega}(2;\one)$. It follows immediately from Lemma \ref{Global_Fourier_Transform_at_trivial_character}  that the $c_\omega$ agree with the factors in Theorem \ref{main_theorem_rtl_points}.
Using $\Gamma(1/2) = \sqrt{\pi}$ gives the statement of Theorem \ref{main_theorem_rtl_points}. \qed
	

\bibliographystyle{alpha}

\end{document}